\documentclass[a4paper,10pt]{amsart}

\usepackage[english]{babel}
\usepackage[utf8]{inputenc}

\usepackage{amssymb}
\usepackage{amsmath}
\usepackage{amsthm}
\usepackage{bbm}

\usepackage{enumerate}
\usepackage{tikz}
\usepackage{marginnote}

\newcommand{\bbN}{\mathbb{N}}

\newcommand{\bbR}{\mathbb{R}}

\newcommand{\bbT}{\mathbb{T}}

\newcommand{\calK}{\mathcal{K}}
\newcommand{\calL}{\mathcal{L}}

\DeclareMathOperator{\id}{id}
\DeclareMathOperator{\one}{\mathbbm{1}}
\DeclareMathOperator{\re}{Re}

\newcommand{\argument}{\mathord{\,\cdot\,}}
\newcommand{\dx}{\;\mathrm{d}}

\DeclareMathOperator{\linSpan}{span}

\newcommand{\norm}[1]{\left\lVert #1 \right\rVert}
\newcommand{\modulus}[1]{\left\lvert #1 \right\rvert}
\DeclareMathOperator{\dom}{dom}
\newcommand{\loc}{\operatorname{loc}}

\newcommand{\spec}{\sigma}
\newcommand{\Res}{\mathcal{R}}

\newcommand{\normEss}[1]{\norm{#1}_{\operatorname{ess}}}

\DeclareMathOperator{\spb}{s}
\DeclareMathOperator{\gbd}{\omega}
\DeclareMathOperator{\Graph}{G}

% ...

\theoremstyle{definition}
\newtheorem{definition}{Definition}[section]

\newtheorem{example}[definition]{Example}
\newtheorem{examples}[definition]{Examples}

\newtheorem{proposition_and_definition}[definition]{Proposition and Definition}
\newtheorem*{further_references}{Further References}

\theoremstyle{plain}
\newtheorem{proposition}[definition]{Proposition}
\newtheorem{lemma}[definition]{Lemma}
\newtheorem{theorem}[definition]{Theorem}
\newtheorem{corollary}[definition]{Corollary}

\numberwithin{equation}{section}

\begin{document}

\title[Positive irreducible semigroups]{Positive irreducible semigroups and their long-time behaviour}
\dedicatory{Dedicated to Rainer Nagel on the occasion of his 80th birthday.}
\author{Wolfgang Arendt}
\address{Wolfgang Arendt, Institut f\"ur Angewandte Analysis, Universität Ulm, D-89069 Ulm, Germany}
\email{wolfgang.arendt@uni-ulm.de}
\author{Jochen Gl\"uck}
\address{Jochen Gl\"uck, Fakultät f\"ur Informatik und Mathematik, Universität Passau, Innstraße 33, D-94032 Passau, Germany}
\email{jochen.glueck@uni-passau.de}
\date{\today}
\begin{abstract}
	The notion~\emph{Perron--Frobenius theory} usually refers to the interaction between three properties of operator semigroups: positivity, spectrum and long-time behaviour. These interactions gives rise to a profound theory with plenty of applications. 
	
	By a brief walk-through of the field and with many examples, we highlight two aspects of the subject, both related to the long-time behaviour of semigroups: (i) The classical question how positivity of a semigroup can be used to prove convergence to an equilibrium as $t \to \infty$. (ii) The more recent phenomenon that positivity itself sometimes occurs only for large $t$, while being absent for smaller times.
\end{abstract}

\maketitle

\section{Introduction}

This article is a brief journey through the world of positive operator semigroups, with a focus on irreducibility and the behaviour for large time. We do not give a comprehensive survey, nor an elementary introduction to the topic; instead, we intend to take the reader on a sight seeing trip to a few highlights of the theory, illustrated by a large variety of examples. We start off with a few basics about the notion of \emph{positivity} (Section~\ref{section:setting-the-stage}), visit a number of \emph{convergence theorems} for $t \to \infty$ -- some of them quite classic, others of more recent descent -- (Section~\ref{section:convergence-to-equilibrium}), make an excursion to \emph{form methods} (Section~\ref{section:forms}), and finish our journey with an outlook to semigroups that are not positive but merely \emph{eventually positive} (Section~\ref{section:ev-pos}). Each theoretical section is accompanied by a selection of applications in the subsequent section. If our readers enjoy the trip, they can find plenty of pointers to related results in paragraphs called \emph{Further References} at the end of most subsections.

\section{Setting the stage: semigroups and positivity}
\label{section:setting-the-stage}

\subsection{Function spaces and Banach lattices}

Throughout we consider operator semigroups on functions spaces such as $L^p$ and $C(K)$, or on Banach spaces with a similar order structure. A common framework for those spaces is provided by the notion \emph{Banach lattice}. Throughout, we assume that $E$ is a Banach lattice over the real field with positive cone $E_+ = \{f \in E: \, f \ge 0\}$. The vectors in $E_+$ are called \emph{positive}. 

The following two notational conventions for vectors $f \in E$ are very convenient: We write $f > 0$ if $f \ge 0$ but $f \not= 0$, and we write $f \gg 0$ if $f$ is a \emph{quasi-interior point} of the positive cone, which means that $\lim_{n \to \infty} u \land (nf) = u$ for each $u \in E_+$. The following examples of Banach lattices occur repeatedly throughout the article:

\begin{examples}[Two important Banach lattices]
	\begin{enumerate}[(a)]
		\item Let $1 \le p \le \infty$ and consider a $\sigma$-finite measure space $(\Omega,\mu)$. Then $E := L^p(\Omega,\mu)$ is a Banach lattice, where a function $f \in E$ is positive iff $f(\omega) \ge 0$ for almost all $\omega \in \Omega$. For a positive function $f$ we have $f > 0$ iff $f$ is not almost everywhere equal to $0$.
			
		In order to understand which elements of $E_+$ are quasi-interior points, one has to distinguish the cases $p < \infty$ and $p = \infty$: If $p < \infty$, then a vector $f \in E_+$ is a quasi-interior point iff $f(\omega) > 0$ for almost all $\omega$. If $p = \infty$, then a vector $f \in E_+$ is a quasi-interior point iff $f \ge \delta \one_\Omega$ for a ($f$-dependent) number $\delta > 0$.

		\item Let $K$ be a compact space -- for instance a bounded and closed subset of $\bbR^d$ -- and let $E = C(K)$ denote the space of real-valued continuous functions on $K$ (endowed with the supremum norm). Then a function $f \in E$ is positive iff $f(x) \ge 0$ for all $x \in K$. A positive function $f$ is a quasi-interior point iff $f \ge \delta \one_K$ for a number $\delta > 0$; by compactness, this is equivalent to the condition $f(x) > 0$ for each $x \in K$.
	\end{enumerate}
\end{examples}

We point out that the cone $E_+$ has empty interior in $E$ for many important Banach lattices -- for instance if $E$ is an infinite-dimensional $L^p$-space with $1 \le p < \infty$.

It is important to note that the dual space $E'$ of a Banach lattice is itself a Banach lattice, whose positive cone $E'_+$ is the set of positive functionals on $E$; here, a functional $\varphi \in E'$ is called positive if $\langle \varphi, f\rangle \ge 0$ for all $f \in E_+$. Finally, we note that a functional $\varphi \in E'$ is called \emph{strictly positive} if $\langle \varphi, f\rangle > 0$ whenever $0 < f \in E$.

\begin{further_references}
	Details about Banach lattices can, for instance, be found in the classical monographs \cite{Schaefer1974, Meyer-Nieberg1991} and in the introductory book \cite{Zaanen1997}.
\end{further_references}

\subsection{Positive semigroups}

A linear operator $T: E \to E$ is called \emph{positive}, which we denote by $T \ge 0$, if $Tf \ge 0$ whenever $f \ge 0$. A positive operator is automatically continuous. 

Now we come to the heart of the matter: operator semigroups. We use the notation $\calL(E)$ for the space of continuous linear operators on our Banach lattice $E$.

\begin{definition}[Semigroups]
	\begin{enumerate}[(a)]
		\item A \emph{one-parameter operator semigroup}, for short a \emph{semigroup}, on $E$ is a mapping $S: (0,\infty) \to \calL(E)$ that satisfies the semigroup law $S(s+t) = S(t)S(s)$ for all $s,t \in (0,\infty)$ and that is locally bounded at $0$, by which we mean $\sup_{t \in (0,1]} \norm{S(t)} < \infty$.
		
		\item A semigroup $S$ is called \emph{strongly continuous} if the orbit map $(0,\infty) \ni t \mapsto S(t)f \in E$ is continuous for each $f \in E$; and $S$ is called a \emph{$C_0$-semigroup} if $S(t)f \to f$ as $t \downarrow 0$ for each $f \in E$.
		
		\item A semigroup $S$ is called \emph{bounded} if $\sup_{t \in (0,\infty)} \norm{S(t)} < \infty$ and it is called \emph{positive} if $S(t)$ is a positive operator for each $t \in (0,\infty)$.
	\end{enumerate}
\end{definition}

It is easy to see that every semigroup $S$ is \emph{exponentially bounded}, i.e., we have $\norm{S(t)} \le M e^{\omega t}$ for every $t \in (0,\infty)$ and constants $M \ge 0$ and $\omega \in \bbR$. The number
\begin{align*}
	\gbd(S) := \inf \{w: \, \exists M \ge 0 \text{ such as } \norm{S(t)} \le Me^{\omega t} \text{ for all } t \in (0,\infty)\} \in [-\infty,\infty)
\end{align*}
is called the \emph{growth bound} of the semigroup $S$.

If $S$ is a $C_0$-semigroup, it has a \emph{generator} $A$, which is defined as the unbounded linear operator $A: E \supseteq \dom(A) \to E$ that acts as $Af := \lim_{t \downarrow 0} \frac{S(t)f - f}{t}$ on the domain
\begin{align*}
	\dom(A) := \{f \in E: \, \lim_{t \downarrow 0} \frac{S(t)f - f}{t} \text{ exists in } E\}.
\end{align*}
For semigroups that are not $C_0$ but have other types of time regularity, one can also give meaning to the notion \emph{generator}; however, instead of discussing various such definitions here, we are going to give references to the literature in those examples where we need generators of semigroups that are not $C_0$.

\begin{further_references}
	\begin{enumerate}[(a)]
		\item For a detailed account of the theory of positive $C_0$-semigroups on Banach lattices, we refer to the classical treatise \cite{Arendt1986} and to the recent book \cite{Batkai2017}.
				
		\item For the related subject of positive semigroups on ordered Banach spaces (rather than only on Banach lattices) we send to the reader to the classical article \cite{Batty1984}. In particular, positive semigroup on $C^*$-algebras and von Neumann algebras occur frequently in mathematical physics due to their applications in quantum theory; we refer for instance to the monograph \cite{Arveson2003} for more details.
	\end{enumerate}
\end{further_references}

\subsection{Irreducibility}

We focus on irreducible semigroups throughout the article. This is a notion that can be defined in various equivalent ways; here is a definition that is very easy from a terminological point of view.

\begin{definition}[Irreducible semigroups]
	A positive semigroup $S$ on our Banach lattice $E$ is called \emph{irreducible} if, for each $0 < f \in E$ and each $0 < \varphi \in E'$, there exists a time $t \in (0,\infty)$ such that
	\begin{align*}
		\langle \varphi, S_t f \rangle > 0.
	\end{align*}
\end{definition}

From an intuitive point of view irreducibility means that, no matter where in the space the initial value $f$ is located, the semigroup moves $f$ through the entire space, so that at some time it meets the functional $\varphi$. Positive semigroups that are not irreducible are sometimes called \emph{reducible}. The following examples are simple but illuminating.

\begin{examples}[Reducible and irreducible semigroups]
	Fix $p \in [1,\infty)$.
	\begin{enumerate}[(a)]
		\item The \emph{left shift semigroup} $S$ on $L^p(\bbR)$ is not irreducible.
		
		\item The \emph{rotation semigroup} $S$ on $L^p(\bbT)$, where $\bbT$ denotes the complex unit circle, is irreducible.
	\end{enumerate}
\end{examples}

For $C_0$-semigroups, irreducibility can be characterised in terms of the resolvent of the generator, see for instance \cite[Proposition~14.10]{Batkai2017}. We will encounter a further criterion for irreducibility in Proposition~\ref{prop:pos-and-irred-by-forms}.

\begin{further_references}
	A recent treatment of irreducible operator semigroups in a quite general setting can be found in \cite{Gao2014}.
\end{further_references}

\section{Convergence to equilibrium}
\label{section:convergence-to-equilibrium}

In this section we discuss a few sufficient criteria for a positive and irreducible semigroup to converge as time tends to infinity. We have to distinguish between \emph{uniform convergence} -- i.e., convergence in operator norm -- and \emph{strong convergence}. As before, $E$ denotes a Banach lattice. Several examples will be discussed in Sections~\ref{section:applications-i} and~\ref{section:applications-ii}.

\subsection{The limit operator}

We first recall a few general facts about the limit operator of a convergent semigroup. By the \emph{fixed space} of a semigroup $S$ on $E$ we mean the set of all vectors $f \in E$ such that $S(t)f = f$ for all times $t$. The \emph{dual semigroup} of a semigroup $S$ on $E$ is the mapping $S': \, (0,\infty) \ni t \mapsto S(t)' \in \calL(E')$.

\begin{proposition} \label{prop:limit-operator}
	Let $S$ be a semigroup on $E$ and assume that, for each $f \in E$, $S(t)f$ converges to a vector $Pf$ as $t \to \infty$. Then:
	\begin{enumerate}[\upshape (a)]
		\item The semigroup $S$ is bounded; the limit operator $P$ is a continuous linear operator on $E$ and a projection (i.e., $P^2 = P$) that commutes with every operator $S(t)$.
		
		\item The range of $P$ is the fixed space of $S$ and the range of the dual operator $P'$ is the fixed space of the dual semigroup $S'$ on $E'$.
		
		\item If $S$ is positive and irreducible, then $P$ is either $0$ or there exist a quasi-interior point $u \in E_+$ and a strictly positive functional $\varphi \in E'$ such that $P = \varphi \otimes u$ (by which we mean $Pf = \langle \varphi, f\rangle u$ for each $f \in E$).
	\end{enumerate}
\end{proposition}

Note that $\langle \varphi, u \rangle = 1$ in assertion~(c) above since $P = \varphi \otimes u$ is a projection.

\begin{proof}[Proof of Proposition~\ref{prop:limit-operator}]
	Assertion (b) is straightforward to prove. The boundedness of $S$ in (a) follows from the local boundedness at $0$ together with the uniform boundedness theorem. Obviously, $P$ commutes with every operator in the semigroup $S$. Moreover, it follows from~(b) that $P$ is a projection.
	
	Assertion~(c) can be shown as follows: the irreducibility and \cite[Proposition~3.11(c)]{Gerlach2019} imply that the fixed space of $S$ is either zero, or one-dimensional and spanned by a quasi-interior point of $E_+$. In the latter case, the dual fixed space is one-dimensional, too, and -- again by the irreducibility -- it is thus spanned by a strictly positive functional.
\end{proof}

Of course, all assertions of Proposition~\ref{prop:limit-operator} hold in particular if $S(t)$ converges uniformly as $t \to \infty$. Moreover, we note the following observation about the speed of convergence in the uniform case.

\begin{proposition} \label{prop:exp-conv-is-uniform}
	Let $S$ be a semigroup on $E$ and assume that $S(t)$ converges uniformly to an operator (hence a projection) $P \in \calL(E)$ as $t \to \infty$. Then there exist real numbers $M \ge 0$ and $\delta > 0$ such that
	\begin{align*}
		\norm{S(t) - P} \le Me^{-\delta t}
	\end{align*}
	for all $t \in (0,\infty)$, i.e., the convergence is exponentially fast.
\end{proposition}
\begin{proof}
	By considering the semigroup $(0,\infty) \ni t \mapsto S(t)(\id_E - P) \in \calL(E)$ we may assume that $P = 0$. Then $S(1)^n \to 0$ as $n \to \infty$, which shows that the spectral radius of $S(1)$ is strictly less than $1$. This, together with the local boundedness of $S$ at the time $0$, implies the assertion.
\end{proof}

\subsection{Uniform convergence}

In the following we discuss a sufficient criterion for an irreducible positive semigroup to converge uniformly as $t \to \infty$. This criterion is a based on an \emph{asymptotic compactness} (or \emph{quasi-compactness}) assumption that is defined as follows.

Let $\calK(E) \subseteq \calL(E)$ denote the space of all compact linear operators on $E$. For every $R \in \calL(E)$ we call
\begin{align*}
	\normEss{R} := \inf \{\norm{R-K}: \, K \in \calK(E)\}
\end{align*}
the \emph{essential norm} of $R$. In fact, $\normEss{\argument}$ is a semi-norm on $\calL(E)$ (and technically speaking, the essential norm of an operator $R \in \calL(E)$ coincides with the quotient norm of the equivalence class of $R$ in the \emph{Calkin algebra} $\calL(E) / \calK(E)$).

\begin{proposition_and_definition}[Asymptotically compact semigroups]
	A semigroup $S$ on $E$ is called \emph{asymptotically compact} (often also \emph{quasi-compact}) if it satisfies the following equivalent assertions:
	\begin{enumerate}[(i)]
		\item $\lim_{t \to \infty} \normEss{S(t)} = 0$.
		\item There exists a times $t \in (0,\infty)$ such that $\normEss{S(t)} < 1$.
	\end{enumerate}
\end{proposition_and_definition}
\begin{proof}[Proof of the equivalence]
	If (ii) holds, then the submultiplicativity of $\normEss{\argument}$ together with the local boundedness of $S$ at $0$ implies~(i).
\end{proof}

Here is a criterion for uniform convergence of an irreducible semigroup. In the version that we present here, no time continuity of the semigroup is required. The theorem is essentially due to Lotz \cite[Theorem~4]{Lotz1986}; we make the modification that we assume $\gbd(S) = 0$ and irreducibility instead of requiring a priori boundedness of the semigroup.

\begin{theorem} \label{thm:lotz-irred}
	Let $S$ be a positive and irreducible semigroup on $E$ with growth bound $\gbd(S) = 0$. If $S$ is asymptotically compact, then there exists a quasi-interior point $u \in E_+$ and a strictly positive functional $\varphi \in E'$ such that $S(t)$ converges with respect to the operator norm to $\varphi \otimes u$ as $t \to \infty$.
\end{theorem}
\begin{proof}[Sketch of proof]
	If the semigroup is bounded, uniform convergence to an operator $P$ as $t \to \infty$ follows from the positivity and the asymptotical compactness; this was proved by Lotz \cite[Theorem~4]{Lotz1986}; the form of the limit operator then follows from Proposition~\ref{prop:limit-operator}.
	
	To show boundedness, one can proceed in three steps: (i) First, one shows by the same arguments as in the proof of \cite[Theorem~4]{Lotz1986} that the spectrum of each operator $S(t)$ intersects the unit circle only in the number $1$. (ii) Fix a time $t_0$; one can show that the order $k$ of the pole $1$ of the resolvent of $S(t_0)$ equals $1$. This can be done by considering the closed $S$-invariant ideal 
	\begin{align*}
		I := \{f \in E: \; (r-1)^k \Res(r,S(t_0))\modulus{f} \to 0 \text{ as } r \downarrow 1\}
	\end{align*}
	and using the irreducibility of $S$. (iii) From the previous two steps we conclude that each operator $S(t_0)$ is power-bounded and hence, the semigroup is bounded.
\end{proof}

The conclusion of the theorem implies that one can decompose the space $E$ as $E = E_1 \oplus E_2$ where $E_1$ is the span of $u$ and where the semigroup tends exponentially to $0$ on $E_2 = \ker \varphi$. For $C_0$-semigroups one can replace the condition $\gbd(S) = 0$ in the theorem with the -- a priori weaker -- condition $\spb(A) = 0$, where $\spb(A)$ denotes the spectral bound of the generator $A$:

\begin{corollary} \label{cor:unif-conv-c_0}
	Let $S$ be a positive and irreducible $C_0$-semigroup on $E$ with generator $A$ and assume that the \emph{spectral bound}
	\begin{align*}
		\spb(A) := \sup \{\re \lambda: \, \lambda \in \spec(A)\}
	\end{align*}
	is equal to $0$. If $S$ is asymptotically compact, then the same conclusion as in Theorem~\ref{thm:lotz-irred} holds.
\end{corollary}

One can derive Corollary~\ref{cor:unif-conv-c_0} from the theorem by using \cite[Theorem~V.3.1]{Engel2000}. Alternatively, one can give a more direct proof of the corollary; such a direct proof can for instance be found in \cite[Theorem~VI.3.5]{Engel2006}.

\begin{further_references}
	\begin{enumerate}[(a)]
		\item Uniform convergence of positive $C_0$-semigroups can be characterised by means of \emph{essential norm continuity}, see \cite[Theorem~3.4]{Thieme1998}. 
		
		\item A criterion for a semigroup to be asymptotically compact is discussed in Section~\ref{section:forms}. Besides, we note that a semigroup $S$ is asymptotically compact iff one (equivalently all) of the operators $S(t)$ has essential spectral radius less than $1$. Sufficient criteria for this latter condition can, for instance, be found in \cite{IonescuTulcea1950}, \cite[Corollaire~1]{Hennion1993}, \cite[Theorem~2 and Corollary~7]{Lotz1986}, \cite[Theorem~2.6]{Martinez1993}, \cite[Theorem~1]{Miclo2015} and \cite[Sections~2 and~4]{Glueck2020}.
		
		\item For a positive $C_0$-semigroup $S$ on $L^p(\Omega,\mu)$ (for a measure space $(\Omega,\mu)$ and $p \in [1,\infty)$) the growth bound $\gbd(S)$ coincides with the spectral bound $\spb(A)$ of the generator $A$. In particular, if $\spb(A) < 0$, then $S_t$ converges uniformly to $0$ as $t \to \infty$. This result is due to Weis \cite{Weis1995, Weis1998}.
	\end{enumerate}
\end{further_references}

\subsection{Strong convergence} \label{subsection:strong-convergence}

In this subsection we recall two sufficient criteria for strong convergence of irreducible semigroups. The first of them requires a considerable amount of time regularity and can be interpreted as a Tauberian theorem. A semigroup $S$ on $E$ is called \emph{eventually norm continuous} if there exists a time $t_0 \in [0,\infty)$ such that the mapping $(t_0,\infty) \ni t \mapsto S(t) \in \calL(E)$ is continuous with respect to the operator norm on $\calL(E)$. If $t_0$ can be chosen as $0$, then we call $S$ \emph{immediately norm continuous}.

\begin{theorem} \label{thm:cyclic-ablv}
	Let $S$ be a positive, bounded and irreducible $C_0$-semigroup with generator $A$ on $E$ and assume that $S$ is eventually norm continuous.  Then precisely one of the following four situations occurs:
	\begin{enumerate}[\upshape (i)]
		\item $\dim (\ker A') = 0$. In this case, $S(t)f \to 0$ as $t \to \infty$ for each $f \in E$.
		
		\item $\dim (\ker A') = 1$ and $\dim (\ker A) = 1$. In this case, there exist a quasi-interior point $u \in E_+$ and a strictly positive functional $\varphi \in E'$ such that $S(t)f \to \langle \varphi, f \rangle u$ as $t \to \infty$ for each $f \in E$.
		
		\item $\dim (\ker A') = 1$ and $\dim (\ker A) = 0$. In this case, $S(t)$ does not converge strongly as $t \to \infty$.
		
		\item $\dim (\ker A') \ge 2$. In this case, $S(t)$ does not converge strongly as $t \to \infty$.
	\end{enumerate}
\end{theorem}
\begin{proof}
	Since the semigroup $S$ is bounded, the dual fixed space $\ker A'$ always separates the fixed space $\ker A$; thus, $\dim(\ker A) \le \dim(\ker A')$. In particular, we are always in one of the four situations~(i)--(iv). Let us now show that the long-lime behaviour is as claimed in each case.
	
	(i) Strong convergence to $0$ in case that $\ker A' = \{0\}$ is, for instance, proved in \cite[Theorem~C-IV-1.5]{Arendt1986}.
	
	(ii) The assumptions imply, by means of \emph{cyclicity} of the peripheral spectrum, that $\spec(A)$ intersects the imaginary axis only in $0$, see \cite[Theorem~C-III-2.10 and Corollary~C-III-2.13]{Arendt1986}. Moreover, the semigroup is mean ergodic since $\ker A$ and $\ker A'$ have the same dimension, so convergence follows from the ABLV theorem (\cite[Theorem~2.4]{Arendt1988} or \cite[Theorem on p.\,39]{Lyubich1988}) by splitting off the range of the mean ergodic projection. The form of the limit operator follows from Proposition~\ref{prop:limit-operator}(c).
	
	(iii)~and~(iv) It follows from Proposition~\ref{prop:limit-operator} that the semigroup cannot be strongly convergent in these cases.
\end{proof}

The situation is simpler when the Banach lattice $E$ has \emph{order continuous norm}; this is for instance the case for every $L^p$-space if $1 \le p < \infty$. For a general definition of the notion \emph{order continuous norm} we refer, e.g., to \cite[Definition~2.4.1]{Meyer-Nieberg1991}.

\begin{corollary} \label{cor:cyclic-ablv-ocn}
	Assume that $E$ has order continuous norm. Let $S$ be a positive, bounded and irreducible $C_0$-semigroup with generator $A$ on $E$ and assume that $S$ is eventually norm continuous.
	
	If $\ker A \not= \{0\}$, then there exist a quasi-interior point $u \in E_+$ and a strictly positive functional $\varphi \in E'$ such that $S(t)f \to \langle \varphi, f \rangle u$ as $t \to \infty$ for each $f \in E$.
\end{corollary}
\begin{proof}
	The assumptions imply that orbits of $S$ are relatively weakly compact in $E$ (see \cite[Proposition~2.2]{Keicher2008}), hence $S$ is mean ergodic. Thus, $\dim(\ker A) = \dim(\ker A')$. One can show, for instance as in the proof of \cite[Proposition~3.11(c)]{Gerlach2019}, that $\ker A'$ contains a non-zero positive element, so we have $\dim(\ker A) = 1$ according to \cite[Proposition~C-III-3.5(c)]{Arendt1986}. Hence, the assertion follows from Theorem~\ref{thm:cyclic-ablv}(ii).
\end{proof}

In the second theorem in this subsection we consider semigroups that dominate a so-called \emph{kernel operator}. If $(\Omega,\mu)$ is a $\sigma$-finite measure space and $E = L^p(\Omega,\mu)$ for $1 \le p < \infty$, then a positive linear operator $T \in \calL(E)$ is called a \emph{kernel operator} if there exists a measurable function $k: \Omega \times \Omega \to [0,\infty)$ such that the following holds for each $f \in E$: the function $k(\omega, \argument) f(\argument)$ is integrable for almost every $\omega \in \Omega$ and the equality
\begin{align*}
	Tf = \int_\Omega k(\argument, \omega) f(\omega) \dx \mu(\omega)
\end{align*}
holds. 

Kernel operators can be characterised abstractly in a purely lattice theoretic way, and this characterization can be used as a definition of kernel operators on more general Banach lattices; see \cite[Proposition~IV.9.8]{Schaefer1974} and \cite[Section~3.3 and Corollary~3.7.3]{Meyer-Nieberg1991} for details. We have the following convergence theorem for semigroups that contain -- or only dominate -- a kernel operator. No time regularity is needed.

\begin{theorem} \label{thm:ggg}
	Assume that $E$ has order continuous norm and let $S$ be a positive, bounded and irreducible semigroup on $E$. Suppose that $S$ has a non-zero fixed point and that there exists a time $t_0 \in (0,\infty)$ and a non-zero kernel operator $K \in \calL(E)$ such that $0 \le K \le S(t_0)$.
	
	Then there exists a quasi-interior point $u \in E_+$ and a strictly positive functional $\varphi \in E'$ such that $S(t)f \to \langle \varphi, f \rangle u$ as $t \to \infty$ for each $f \in E$.
\end{theorem}
\begin{proof}
	Strong convergence under these assumptions was proved in \cite[Corollary~4.4]{Gerlach2019}. The form of the limit operator follows from Proposition~\ref{prop:limit-operator}.
\end{proof}

\begin{further_references}
	\begin{enumerate}[(a)]
		\item The \emph{cyclicity of the peripheral spectrum} on which Theorem~\ref{thm:cyclic-ablv} is based has a rather long history, going back to the origins of Perron--Frobenius theory at the beginning of the 20th centure. For positive operators on Banach lattices, a quite general cyclicity result was first proved by Lotz \cite[Section~4]{Lotz1968}. For semigroups, such a result is due to Derndinger \cite[Theorem~3.7]{Derndinger1980} (in a special case) and Greiner \cite[Theorem~2.4]{Greiner1981} (in the general case).
		
		\item The asymptotic theory of semigroups that contain or dominate a kernel operator follows essentially two lines of development:
		\begin{itemize}
			\item For semigroups of kernel operators on $L^p$-spaces, a first convergence result is due to Greiner \cite[Korollar~3.11]{Greiner1982a}. This was followed by a series of papers and results in \cite[Theorem~12]{Davies2005}, \cite[Section~4]{Arendt2008}, \cite[Theorem~4.2]{Gerlach2013}, \cite{Gerlach2017} and \cite[Sections~3 and~4]{Gerlach2019}.
			
			\item For the important case of Markovian $C_0$-semigroups on $L^1$-spaces, Theorem~\ref{thm:ggg} is due to Pich\'or and Rudnicki \cite[Theorems~1 and~2]{Pichor2000}. Recent generalizations and adaptations of this result can be found in \cite[Theorem~2, Corollary~2 and Proposition~2]{Pichor2016}, \cite[Theorem~2.2]{Pichor2018} and \cite[Theorem~2.1]{Pichor2018a}.
		\end{itemize}
	\end{enumerate}
\end{further_references}

\section{Applications I}
\label{section:applications-i}

\subsection{Diffusion with nonlocal boundary conditions}

In the following example of the Laplace operator with nonlocal boundary conditions, which was studied in detail in \cite{Arendt2016} and in \cite[Chapter~3]{Kunkel2017}, Theorem~\ref{thm:lotz-irred} can be applied.

\begin{example}
	\label{ex:nonlocal-boundary-conditions}
	Let $\emptyset \not= \Omega \subseteq \bbR^d$ be open, bounded and connected and suppose that $\Omega$ has Lipschitz boundary. For each $z \in \partial \Omega$, let $\mu(z)$ be a Borel probability measure on $\Omega$ such that the mapping
	\begin{align*}
		\partial \Omega \ni z \mapsto \int_\Omega f(x) \mu(z, \dx x) \in \bbR
	\end{align*}
	is continuous for each bounded continuous function $f: \Omega \to \bbR$. We consider the Laplace operator $\Delta_\mu$ on $C(\overline{\Omega})$ with domain
	\begin{align*}
		\dom(\Delta_\mu) = \left\{ u \in C(\overline{\Omega}): \, \Delta u \in C(\overline{\Omega}) \, \text{ and } \, u(z) = \int_\Omega u(x) \mu(z,\dx x) \right\}.
	\end{align*}
	These boundary conditions are nonlocal since the behaviour of $u$ at the boundary is related to its behaviour in $\Omega$ for $u \in \dom(\Delta_\mu)$. The boundary conditions have the following probabilistic interpretation (which is also explained in \cite[p.\,2484]{Arendt2016}): we consider a Brownian motion in $\Omega$, and whenever a particle reaches the boundary at a point $z \in \partial \Omega$, it is immediately transported back into a position within $\Omega$, which is determined by the probability distribution $\mu(z,\argument)$.
	
	It is proved in \cite[Theorem~1.3]{Arendt2016} that $\Delta_\mu$ generates an immediately norm continuous (even holomorphic), positive and contractive semigroup $S$ on $C(\overline{\Omega})$ and each operator $S(t)$ is compact. Uniform convergence of the semigroup as $t \to \infty$ was also shown in \cite[Theorem~1.3]{Arendt2016}; in the following, we explain how this is related to Theorem~\ref{thm:lotz-irred}.
	
	We first note that the semigroup $S$ is not a $C_0$-semigroup. (One can still speak of a generator, though; see for instance \cite[Section~2]{Arendt2016} or \cite[Definition 3.2.5]{Arendt2011}). Since $\Omega$ is connected, the semigroup $S$ is irreducible, see \cite[Proposition~3.29]{Kunkel2017}. Moreover, the constant function $\one$ is a fixed vector of the semigroup since we assumed each measure $\mu(z)$ to be a probability measure; so the semigroup has growth bound $0$.
	
	It thus follows from Theorem~\ref{thm:lotz-irred} that $S_t$ converges with respect to the operator norm to $\nu \otimes \one$ as $t \to \infty$, for some Borel probability measure $\nu$ on $\overline{\Omega}$ which is supported everywhere on $\Omega$.
\end{example}

\begin{further_references}
	\begin{enumerate}[(a)]
		\item We point out that the Laplace operator in Example~\ref{ex:nonlocal-boundary-conditions} can be replaced with much more general elliptic operators; see \cite{Arendt2016} and \cite[Section~3]{Kunkel2017}. The same references also show that the condition that $\Omega$ have Lipschitz boundary can be somewhat relaxed.
		
		\item One can also study non-local Robin boundary conditions instead of non-local Dirichlet boundary conditions; see \cite[Section~4]{Kunkel2017} and \cite{Arendt2018}.
		
		\item Similar questions as in Example~\ref{ex:nonlocal-boundary-conditions} can also be studied on unbounded domains; this is the content of the recent article \cite{Kunze2020}.
	\end{enumerate}
\end{further_references}

\subsection{Strong convergence for Schrödinger semigroups on $L^1(\bbR^d)$}

We give an example for an application of Theorem~\ref{thm:cyclic-ablv}; the semigroups that occur in this example are discussed in more detail in \cite{Arendt1992a}.

\begin{example}
	\label{ex:schroedinger-sg-on-l_1}
	Let $0 \le m \in L^1_{\loc}(\bbR^d)$ and define the operator $A$ on $L^1(\bbR^d)$ by
	\begin{align*}
		\dom(A) & = \{u \in L^1(\bbR^d): \, \Delta u \in L^1(\bbR^d) \text{ and } mu \in L^1(\bbR^d) \}, \\
		Au & = \Delta u - mu.
	\end{align*}
	Then $A$ is the generator of a positive, irreducible, contractive and immediately norm continuous $C_0$-semigroup $S$ on $L^1(\bbR^d)$ that is dominated by the Gaussian semigroup $T$ on $L^1(\bbR^d)$; see \cite[Part~A]{Kato1986}. The domination of $S$ by $T$ implies that $\ker A = \{0\}$. Indeed if $f \in \ker A$, then
	\begin{align*}
		\modulus{f} = \modulus{S(t)f} \le S(t) \modulus{f} \le T(t) \modulus{f}
	\end{align*}
	for each $t \in (0,\infty)$. But the Gaussian semigroup $T$ is norm-preserving on the positive cone of $L^1(\bbR^d)$, so $\modulus{f}$ is a fixed vector of $T$, hence $f = 0$.
	
	So according to Theorem~\ref{thm:cyclic-ablv} the asymptotic behaviour of $S$ depends merely on the question whether $\ker A'$ is zero or non-zero. Let us discuss this in two particular cases:
	\begin{enumerate}[(a)]
		\item If the dimension $d$ is $1$ or $2$ and $m \not= 0$, then $\ker A' = \{0\}$. So, no matter how small the (non-zero) absorption term $m$ is, we have $S(t)f \to 0$ as $t \to \infty$ for each $f \in L^1(\bbR^d)$. For details, see \cite[Theorem~3.2]{Arendt1992a}
		
		\item If $d \ge 3$ and $\int_{\modulus{y} \ge 1} \frac{m(y)}{\modulus{y}^{d-2}} \dx y < \infty$, then $\ker A' \not= \{0\}$. Thus, the semigroup $S$ is not strongly convergent as $t \to \infty$. For details, we refer to \cite[Theorem~3.6]{Arendt1992a}. \medskip
	\end{enumerate}
\end{example}

In the above example we only had the cases where the semigroup is either convergent to $0$ or not strongly convergent at all. A Schrödinger semigroup (on $L^2$ instead of $L^1$) where one has strong convergence to a non-zero equilibrium will be discussed in Example~\ref{ex:schroedinger-on-l_2-strong-convergence}.

\begin{further_references}
	\begin{enumerate}[(a)]
		\item  Another way to place the results from~\cite{Arendt1992a} (from which Example~\ref{ex:schroedinger-sg-on-l_1} above is taken) in the abstract context of positive operators on vector lattices is presented in \cite{Teichmann1997}.
		
		\item For the important case of Markovian $C_0$-semigroups on $L^1$-spaces, Theorem~\ref{thm:ggg}, and various versions thereof, turn out to be extremely useful in mathematical biology. See for instance the recent papers \cite{Pichor2018a, Pichor2018} and the plenty of references therein.
		
		\item A related version of Theorem~\ref{thm:ggg} in \cite[Theorem~3.5]{Gerlach2019} is instrumental to prove convergence results for so-called \emph{transition semigroups} on spaces of measures, which occur frequently in stochastic analysis and in PDE theory and which are in many cases not $C_0$. For details we refer to the recent preprint \cite{GerlachTransition}.
	\end{enumerate}
\end{further_references}

\section{Semigroups associated with forms, and criteria for asymptotic compactness}
\label{section:forms}

An important class of semigroups are those associated with a bilinear form. In fact, many classical parabolic problems are governed by forms. Let $H$ be a Hilbert space over the real field.

\begin{definition}[Closed forms]
	A \emph{closed form} on $H$ is a pair $(a,V)$ with the following properties:
	\begin{enumerate}[(a)]
		\item $V$ is a Hilbert space that is densely and continuously embedded in $H$ (for short: $V \overset{\operatorname{d}}{\hookrightarrow} H$).
		
		\item $a: V \times V \to \bbR$ is a bilinear form which is \emph{continuous} (i.e., $\modulus{a(u,v)} \le M \norm{u}_V \norm{v}_V$ for all $u,v \in V$ and a fixed constant $M \ge 0$) and \emph{elliptic}, i.e., there exist $\omega \ge 0$ and $\alpha > 0$ such that
		\begin{align}
			\label{eq:elliptic-form}
			a(u,u) + \omega \norm{u}_H^2 \ge \alpha \norm{u}_V^2 \qquad \text{for all } u \in V.
		\end{align}
	\end{enumerate}
\end{definition}

Let $(a,V)$ be a closed form on $H$. Since $V$ is dense in $H$ there exists a unique operator $A$ on $H$ whose graph is given by $\Graph(A) = \{(u,f): \, u \in V, \; f \in H, \text{ and } a(u,v) = -\langle f,v\rangle_H \text{ for all } v \in V\}$. This operator $A$ generates of $C_0$-semigroup $S$ on $H$ which is immediately norm continuous (and even holomorphic). We say that the form $a$ is \emph{positive-coercive} if~\eqref{eq:elliptic-form} holds with $\omega = 0$.

If the embedding of $V$ into $H$ is not only continuous but even compact, then the operator $A$ has compact resolvent and each $S(t)$ is compact.

\begin{proposition}
	Let $(a,V)$ be a closed form on $H$.
	\begin{enumerate}[(a)]
		\item If $a$ is \emph{accretive}, i.e., $a(u,u) \ge 0$ for all $u \in V$, then $\norm{S(t)} \le 1$ for all $t \in (0,\infty)$.
		
		\item If $a$ is positive-coercive, then there exists a number $\delta > 0$ such that $\norm{S(t)} \le e^{-\delta t}$ for all $t \in (0,\infty)$.
	\end{enumerate}
\end{proposition}
\begin{proof}
	(a) If $a$ is accretive, then $A$ is dissipative and hence, $S$ is contractive \cite[Theorem~II.3.15]{Engel2000}.
	
	(b) There exists a constant $c_H > 0$ such that $\norm{u}_H^2 \le c_H \norm{u}_V^2$ for all $u\in V$. Therefore,	\begin{align*}
		- \langle Au,u\rangle_H = a(u,u) \ge \alpha \norm{u}_V^2 \ge \frac{\alpha}{c_H} \norm{u}_H^2.
	\end{align*}
	for $u \in \dom(A)$. If we set $\delta := \alpha / c_H$, we thus have $\langle (A+\delta)u, u \rangle_H \le 0$ for each $u \in \dom(A)$. So $A+\delta$ is dissipative and therefore generates a contractive semigroup. Hence, $\norm{e^{\delta t} S(t)} \le 1$ for all times $t$.
\end{proof}

If we relax the coercivity condition in an appropriate way we obtain an asymptotically compact semigroup. To this end, we use the following notion from \cite[Definition~4.2(b)]{Arendt2020}.

\begin{definition}[Essential positive-coercivity of a form]
	Let $V \overset{\operatorname{d}}{\hookrightarrow} H$ and let $a: V \times V \to \bbR$ be a bilinear mapping. Then $a$ is called \emph{essentially positive-coercive}	if $\norm{u_n}_V \to 0$ for every sequence $(u_n)$ in $V$ that satisfies $u_n \rightharpoonup 0$ in $V$ and $\limsup_{n \to \infty} a(u_n,u_n) \le 0$.
\end{definition}

Here, we use the symbol $u_n \rightharpoonup 0$ to denote \emph{weak convergence} to $0$ in $V$, i.e., $\langle u_n,v \rangle_V \to 0$ for each $v \in V$. Clearly, positive-coercivity implies essential positive-coercivity.

\begin{theorem}
	\label{thm:asymp-compact-via-forms}
	Let $V \overset{\operatorname{d}}{\hookrightarrow} H$ and let $a: V \times V \to \bbR$ be bilinear, continuous and essentially positive coercive. Then $(a,V)$ is closed and the associated $C_0$-semigroup $S$ is asymptotically compact.
\end{theorem}
\begin{proof}
	This has recently been proved in \cite[Proposition~6.1 and Theorem~6.3]{Arendt2020}.
\end{proof}

Thus, if we are in the situation of Theorem~\ref{thm:asymp-compact-via-forms} and know in addition that $S$ is positive and irreducible and that the largest spectral value of the associated operator $A$ is $0$, then $S(t)$ converges in operator norm to a rank-$1$ operator as $t \to \infty$ (by Corollary~\ref{cor:unif-conv-c_0}).

The following special case of the Beurling--Deny--Ouhabaz criterion is very helpful in establishing positivity and irreducibility in many concrete situations.

\begin{proposition}
	\label{prop:pos-and-irred-by-forms}
	Let $(\Omega,\mu)$ be a $\sigma$-finite measure space and $H = L^2(\Omega)$. Let $(a,V)$ be a closed form on $H$ and denote the associated $C_0$-semigroup on $H$ by $S$.
	\begin{enumerate}[(a)]
		\item If $u \in V$ implies $u^+ \in V$ and $a(u^+,u^-) \ge 0$, then $S$ is positive.
		
		\item Assume that $S$ is positive. If for each measurable set $B \subseteq \Omega$ the condition $\one_B V \subseteq V$ implies $\mu(B) = 0$ or $\mu(\Omega \setminus B) = 0$, then $S$ is irreducible.
	\end{enumerate}
\end{proposition}

In the above proposition, we used the notation $u^+$ for the pointwise supremum of $u$ and $0$ and the notation $u^-$ for the function $(-u)^+ = u^+ - u$.

In all our applications, $\Omega$ will be a (non-empty) open subset of $\bbR^d$ (with the Lebesgue measure) and $V$ will be a subspace of the first Sobolev space
\begin{align*}
	H^1(\Omega) := \{u \in L^2(\Omega): \, \partial_j u \in L^2(\Omega) \text{ for all } j = 1,\dots, d\},
\end{align*}
where we understand the derivative $\partial_j u$ in the sense of distributions. The space $H^1(\Omega)$ is a \emph{sublattice} of $L^2(\Omega)$, which means that $u^+,u^- \in H^1(\Omega)$ whenever $u \in H^1(\Omega)$. In fact, we have $\partial_j (u^+) = \one_{\{\omega \in \Omega: \, u(\omega) > 0\}} \partial_j u$ for each $u \in H^1(\Omega)$ \cite[Proposition~4.4]{Ouhabaz2005}.

The following lemma, which we quote from \cite[Lemma~11.1.1]{Arendt2005}, is the key to apply Proposition~\ref{prop:pos-and-irred-by-forms}(b) to forms defined on subspaces of $H^1(\Omega)$.

\begin{lemma}
	\label{lem:sobolev-space-no-band-projections}
	Let $\Omega \subseteq \bbR^d$ be non-empty, open and connected, and let $B \subseteq \Omega$ be Borel measurable. Assume that $\one_B v \in H^1(\Omega)$ for all test functions $v$ on $\Omega$. Then $\lambda(B) = 0$ or $\lambda(\Omega \setminus B) = 0$ (where $\lambda$ denotes the Lebesgue measure on $\Omega$).
\end{lemma}

\begin{further_references}
	Form methods are an excellent tool for the study of heat equations on subsets of $\bbR^d$. For details we refer to the comprehensive monograph \cite{Ouhabaz2005}.
\end{further_references}

\section{Applications II}
\label{section:applications-ii}

\subsection{Schrödinger semigroups on $L^2(\bbR^d)$: asymptotic compactness}

It is well-known in mathematical physics that the essential spectrum of a Schrödinger operator $\Delta + m$ on $L^2(\bbR^d)$ (where $m$ denotes a potential) is closely related to the behaviour of $m(x)$ for large $\modulus{x}$. In the following example we demonstrate how Theorem~\ref{thm:asymp-compact-via-forms} provides one method to see this.

\begin{example}[Schrödinger semigroups on $L^2(\bbR^d)$]
	\label{ex:schroedinger-on-l_2-asymptotic-compactness}
	Let $H = L^2(\bbR^d)$ for dimension $d \ge 3$ and fix a function $m \in L^r_{\loc}(\bbR^d)$, where $r > \frac{d}{2}$. We assume that $m$ satisfies $\liminf_{\modulus{x} \to \infty} m(x) > 0$. Consider the subspace $V := \{u \in H^1(\bbR^d): \, \int_{\bbR^d} \modulus{m} u^2 < \infty\}$ of $H^1(\bbR^d)$. We note that $H^1(\bbR^d) \subseteq L^{\frac{2d}{d-2}}(\bbR^d)$, so $mu \in L^1_{\loc}(\bbR^d)$ for all $u \in V$. Define an operator $A: L^2(\bbR^d) \supseteq \dom(A) \to L^2(\bbR^d)$ by
	\begin{align*}
		\dom(A) & = \{u \in V: \, \Delta u - mu \in L^2(\bbR^d)\}, \\
		Au & = \Delta u - mu,
	\end{align*}
	where we use that $\Delta u - mu$ is a distribution for all $u \in V$.
	
	\emph{Claim:} The operator $A$ generates a positive and irreducible $C_0$-semigroup $S$ on $L^2(\bbR^d)$ which is asymptotically compact. Thus, if $\spb(A) = 0$, then $S(t)$ converges uniformly to $w \otimes w$ as $t \to \infty$, where $w \in L^2(\bbR)$ is a function of norm $1$ that is $>0$ almost everywhere.
\end{example}
\begin{proof}
	The space $V$ is a Hilbert space with respect to the norm $\norm{\argument}_V$ given by
	\begin{align*}
		\norm{u}_V^2 = \norm{u}_{H^1}^2 + \int_{\bbR^d} \modulus{m} u^2 = \int_{\bbR^d} u^2 + \int_{\bbR^d} \modulus{\nabla u}^2 + \int_{\bbR^d} \modulus{m} u^2.
	\end{align*}
	Note that $V$ is densely embedded in $L^2(\bbR^d)$. We now define a form $a: V \times V \to \bbR$ by
	\begin{align*}
		a(u,v) = \int_{\bbR^d} \nabla u \cdot \nabla v  +  \int_{\bbR^d} muv.
	\end{align*}
	Then $a$ is continuous. We show that~(a) the form $a$ is essentially positive-coercive, (b) its associated operator coincides with $A$ and~(c) the generated semigroup has the desired properties.
	
	\begin{enumerate}[(a)]
		\item We show that $a$ is essentially positive-coercive. To this end, let $u_n \rightharpoonup 0$ in $V$ and assume that $\limsup_{n \to \infty} a(u_n,u_n) \le 0$. We first note that this implies $u_n \rightharpoonup 0$ in $H^1(\bbR^d)$.
		
		Next, we observe that there exists $0 < \delta \le 1$ and a ball $B \subseteq \bbR^d$ such that $m(x) \ge \delta$ for $x$ in the complement  of $B$. We have
		\begin{align*}
			\norm{u_n}_V^2 = \int_{\bbR^d \setminus B} u_n^2 \, + \, \int_B u_n^2 \quad + \quad \int_{\bbR^d} \modulus{\nabla u_n}^2 \quad + \quad \int_{\bbR^d \setminus B} m u_n^2 \, + \, \int_{B} \modulus{m} u_n^2,
		\end{align*}
		so we have to prove that all five terms in the sum converge to $0$. Choose $r'$ conjugate to $r$, i.e., $\frac{1}{r} + \frac{1}{r'} = 1$. Then
		\begin{align*}
			\modulus{\int_B m u_n^2} \le \int_B \modulus{m} u_n^2 \le \norm{m}_{L^r(B)} \, \norm{u_n}_{L^{2r'}(B)}^2,
		\end{align*}
		and the latter sequence converges to $0$ since the embedding $H^1(B)  \hookrightarrow L^{2r'}(B)$ is compact (as $2r' < \frac{2d}{d-2}$). Consequently, our assumption $\limsup_{n \to \infty} a(u_n,u_n) \le 0$ implies that
		\begin{align*}
			\int_{\bbR^d} \modulus{\nabla u_n}^2 + \int_{\bbR^d \setminus B} m u_n^2 \to 0
		\end{align*}
		Since $m(x) \ge \delta$ for $x \in \bbR^d \setminus B$, it follows that also $\int_{\bbR^d \setminus B} u_n^2 \to 0$. The only remaining term to deal with is $\int_B u_n^2$; it converges to $0$, too, since the embedding $H^1(B) \hookrightarrow L^2(B)$ is compact. Consequently, we indeed have $\norm{u_n}_V \to 0$.
		
		\item By Theorem~\ref{thm:asymp-compact-via-forms} the form $(a,V)$ is closed. Let $\tilde A$ denote the semigroup generator associated with this form. We intend to show that $\tilde A = A$, and in order to do so, we first prove that the space $C_c^\infty(\bbR^d)$ of all test functions is dense in $V$. So let $u \in V$.
		
		\emph{Step 1:} There exists a sequence $(v_n) \subseteq C_c^\infty(\bbR^d)$ that converges to $u$ in $H^1(\bbR^d)$ and pointwise almost everywhere. Since $H^1(\bbR^d)$ is a sublattice of $L^2(\bbR^d)$ and the lattice operations are continuous, it follows that
		\begin{align*}
			u_n := \big( v_n \lor -\modulus{u} \big) \land \modulus{u}
		\end{align*}
		also converges to $u$ in $H^1(\bbR^d)$ and pointwise almost everywhere. The dominated convergence theorem thus shows that $u_n$ converges to $u$ in $V$.
		
		\emph{Step 2:} To conclude that the test functions are dense in $V$, it remains to show that each function $u_n$ can itself be approximated by test functions. So fix an index $n$ and set $w := u_n$. There exists a ball $B \subseteq \bbR^d$ with center $0$ such that $w$ vanishes outside of $\frac{1}{2}B$. We choose a sequence of mollifiers $(\rho_k)$ and define $w_k := \rho_k \star w$ for each index $k$. For all sufficiently large $k$ we have $w_k \in C_c^\infty(B)$, and the sequence $(w_k)$ converges to $w$ in $H^1(B)$; consequently, $w_k \to w$ in $L^{\frac{2d}{d-2}}(B)$. Using H\"older's inequality, we now obtain
		\begin{align*}
			\int_B \modulus{w_k-w}^2 \modulus{m} \le \norm{w_k-w}_{L^{\frac{2d}{d-2}}(B)}^2 \norm{m}_{L^{\frac{d}{2}}(B)} \to 0.
		\end{align*}
		Hence, $w_k \to w$ in $V$. Therefore, the test functions are indeed dense in $V$.
		
		\emph{Step 3:} Now we can show that $\tilde A = A$. Indeed, let $u \in \dom(\tilde A)$ and set $f := \tilde A u$. For all $v \in V$ we then have
		\begin{align}
			\label{eq:schroedinger-equality-of-operators}
			\int_{\bbR^d} \nabla u \cdot \nabla v + \int_{\bbR^d} muv = - \int_{\bbR^d} fv.
		\end{align}
		For each test function $v$ we have
		\begin{align*}
			\langle \Delta u, v\rangle = - \int_{\bbR^d} \nabla u \cdot \nabla v = \int_{\bbR^d} (mu + f)v,
		\end{align*}
		so $u \in \dom(A)$ and $Au = f$. Now assume conversely that $u \in \dom(A)$ and set $f := Au$. Then~\eqref{eq:schroedinger-equality-of-operators} holds for all $v \in C_c^\infty(\bbR^d)$. Since we have shown above that the test functions are dense in $V$, we conclude that~\eqref{eq:schroedinger-equality-of-operators} even holds for all $v \in V$. Thus, $u \in \dom(\tilde A)$ and $\tilde A u = Au$.
		
		We have shown that $\tilde A = A$.
		
		\item Since the operator $A$ equals $\tilde A$, it generates a norm-continuous (even holomorphic) $C_0$-semigroup $S$ on $L^2(\bbR^d)$. This semigroup is asymptotically compact due to Theorem~\ref{thm:asymp-compact-via-forms}. It thus remains to show that $S$ is positive and irreducible. To this end we will use Proposition~\ref{prop:pos-and-irred-by-forms}.
		
		So let $u \in V$. Then $u^+ \in H^1(\bbR^d)$. Since $(u^+)^2 \le u^2$, one has $u^+ \in V$. Moreover, since $u^+u^- = 0$ and, by Stampacchia's lemma, $\partial_j u^+ \partial_j u^- = 0$, it follows that $a(u^+,u^-) = 0$. Hence, $S$ is positive by Proposition~\ref{prop:pos-and-irred-by-forms}(a).
		
		To show irreducibility, let $B \subseteq \bbR^d$ be a Borel set such that both $B$ and $\bbR^d \setminus B$ have strictly positive Lebesgue measure. According to Lemma~\ref{lem:sobolev-space-no-band-projections} there exists a test function $v \in C_c^\infty(\bbR^d)$ such that $\one_B v \not\in H^1(\bbR^d)$. Thus, $v \in V$ but $\one_B v \not\in V$, so irreducibility follows from Proposition~\ref{prop:pos-and-irred-by-forms}(b).
	\end{enumerate}
	Convergence of the semigroup in case that $\spb(A) = 0$ now follows from Corollary~\ref{cor:unif-conv-c_0}, and the fact that the limit operator is symmetric follows since $A$ is self-adjoint (as the form $a$ is symmetric).
\end{proof}

\begin{further_references}
	For more information about the closely related question where the essential spectrum of a Schr\"odinger operator is located, we refer to the literature in mathematical physics, for instance to the classical reference \cite[Chapter~4]{Reed1978}.
\end{further_references}

\subsection{A Schrödinger semigroup on $L^2(\bbR)$: strong convergence}

The semigroup on $L^2(\bbR^d)$ generated by the Laplace operator $\Delta$ without potential converges strongly to $0$ as $t \to \infty$, while we have operator norm convergence to a non-zero equilibrium in Example~\ref{ex:schroedinger-on-l_2-asymptotic-compactness} (if $\spb(A) = 0$). The following example, which we owe to Mateusz Kwa\'snicki \cite{Kwasnicki2019}, shows that strong convergence to a non-zero equilibrium can also occur.

\begin{example}
	\label{ex:schroedinger-on-l_2-strong-convergence}
	Define $m \in L^\infty(\bbR)$ by $m(x) = \frac{6x^2 - 2}{(1+x^2)^2}$ for $x \in \bbR$, and consider the Schr\"odinger operator $A: L^2(\bbR) \supseteq H^2(\bbR) \to L^2(\bbR)$ that is given by $Au = \Delta u - mu$. Then $A$ is self-adjoint with finite spectral bound and thus generates an eventually norm continuous (even holomorphic) $C_0$-semigroup $S$ on $L^2(\bbR)$. It follows from perturbation theory for positive semigroups that $S$ is positive and irreducible (see for instance \cite[Proposition~C-III-3.3]{Arendt1986}). A direct computation shows that the function $w \in L^2(\bbR)$ given by
	\begin{align*}
		w(x) = \frac{1}{1+x^2} \quad \text{for all } x \in \bbR
	\end{align*}
	is in the kernel of $A$. Next we note that $\spb(A) = 0$: since $m(x) \to 0$ as $\modulus{x} \to \infty$, it follows that the essential spectrum of $A$ coincides with $(-\infty,0]$ (see for instance \cite[Corollary~2 on p.\,113 and Example~8 on p.\,118]{Reed1978}; this is also related to Example~\ref{ex:schroedinger-on-l_2-asymptotic-compactness} above). Hence, if $\spb(A) > 0$, then $\spb(A)$ is an isolated eigenvalue of $A$; the positivity of $S$ implies that there is a corresponding eigenvector $v > 0$, and the irreducibility thus yields that $v \gg 0$ \cite[Proposition~C-III-3.5(a)]{Arendt1986}. Consequently,
	\begin{align*}
		\langle v, w \rangle = \langle v, S(t)w \rangle = \langle S(t)v, w \rangle = e^{t\spb(A)} \langle v,w \rangle
	\end{align*}
	for each time $t \in (0,\infty)$. This is a contradiction since $\langle v, w\rangle \not = 0$. Hence, $\spb(A) = 0$.
	
	We can thus conclude from Corollary~\ref{cor:cyclic-ablv-ocn} (or from Theorem~\ref{thm:ggg}) that $S(t)$ converges strongly to a multiple of $w \otimes w$ as $t \to \infty$.
\end{example}

The convergence of the semigroup $S$ in the above example can, of course, also be shown by simpler methods than Corollary~\ref{cor:cyclic-ablv-ocn} or Theorem~\ref{thm:ggg} -- for instance by employing the spectral theorem for self-adjoint operators. Yet, we chose to include this simple example since it gives a nice contrast to Example~\ref{ex:schroedinger-on-l_2-asymptotic-compactness} and since it is a good illustration of the use of irreducibility (to conclude that there are no eigenvalues in $(0,\infty)$).

\section{Eventual positivity}
\label{section:ev-pos}

While positivity of semigroups is a ubiquitous phenomenon in analysis, it might come as a surprise at first glance that some evolution equations exhibit only \emph{eventually positive} rather than positive behaviour. We say that a semigroup $S$ on our Banach lattice $E$ is \emph{eventually positive} if there exists a time $t_0 > 0$ such that $S(t) \ge 0$ for all $t \ge t_0$; this does not assume (nor does it imply) that $S(t)$ is positive for small times. In fact, the property that we just described should more precisely be called \emph{uniform} eventual positivity, since there exists a single time from which on all operators are positive -- in contrast to the situation where the single orbits become eventually positive for positive initial value, but where the time when this happens depends on the initial value (see \cite[Example~5.7]{Daners2016} for an example which shows that this can indeed occur).

Here is a simple criterion for eventual positivity of self-adjoint semigroups.

\begin{theorem}
	\label{thm:ev-pos}
	Let $S$ be a self-adjoint $C_0$-semigroup on $L^2 := L^2(\Omega,\mu)$ for a $\sigma$-finite measure space $(\Omega,\mu)$. Assume that there exists a quasi-interior point $u$ in $L^2$ with the following two properties:
	\begin{enumerate}[(1)]
		\item There exists a time $t_0$ such that, for each $f \in L^2$, the modulus $\modulus{S(t_0)f}$ is dominated by an ($f$-dependent) multiple of $u$.
		
		\item The largest spectral value of the generator $A$ is an eigenvalue whose eigenspace is spanned by a function $w$ that satisfies $w \ge c u$ for a number $c > 0$.
	\end{enumerate}
	Then there exists $t_1 > 0$ such that $S(t) \ge 0$ for all $t \ge t_1$. In fact, $t_1$ can be chosen such that, for each $0 < f \in L^2(\Omega,\mu)$ and each $t \ge t_1$, the function $S(t)f$ even dominates a strictly positive multiple of $u$.
\end{theorem}
\begin{proof}
	This was proved in \cite[Corollary~3.5]{DanersUnif}.
\end{proof}

The property that, for all sufficiently large times $t$ and all non-zero $f \ge 0$, the function $S(t)f$ is not only positive, but dominates a non-zero multiple of $u$, can be interpreted as a rather strong irreducibility property.

\begin{further_references}
	\begin{enumerate}[(a)]
		\item Eventually positive behaviour in infinite dimensions has already been observed in 2008 in \cite{Ferrero2008, Gazzola2008} for the biharmonic heat equation on $\bbR^d$ (compare also the recent article \cite{Ferreira2019}). A few years later, a case study for the Dirichlet-to-Neumann semigroup on the unit circle was presented in \cite{Daners2014}, which then set off the development of a general theory.
		
		\item The present state of the art for eventually positive semigroups can be found in the articles \cite{Daners2016, Daners2016a} where the basic theory is presented, and in \cite{Daners2017, DanersPERT, DanersUnif} where various aspects of the theory are further developed.
		
		\item In finite dimensions, eventual positivity of matrices and matrix semigroups has been considered several years earlier than in infinite dimensions; we refer for instance to \cite{Tarazaga2001, Noutsos2006, Noutsos2008} (which are only three examples out of a wealth of articles in this field). 
		
		\item A closely related topic concerns \emph{eventual domination} of semigroups, which is treated in \cite{GlueckEvDom}.
	\end{enumerate}
\end{further_references}

\section{Applications III}
\label{section:applications-iii}

\subsection{A Laplace operator with non-local boundary conditions on $L^2(0,1)$}

A very simple example where eventual positivity occurs is the Laplace operator on $L^2(0,1)$ with a special type of non-local boundary conditions. In contrast to Example~\ref{ex:nonlocal-boundary-conditions} we do not couple the boundary behaviour with the behaviour in the interior, now. Instead, we impose a coupling between the boundary conditions at the two endpoints of the interval $(0,1)$.

\begin{example} \label{ex:ev-pos-non-local-laplace}
	Let $S$ denote the semigroup on $L^2(0,1)$ that is associated with the form $a: H^1(0,1) \times H^1(0,1) \to \bbR$ given by 
	\begin{align*}
		a(u,v) = \int u' v' +
		\begin{pmatrix}
			u(0) & u(1)
		\end{pmatrix}
		\begin{pmatrix}
			1 & 1 \\
			1 & 1
		\end{pmatrix}
		\begin{pmatrix}
			v(0) \\ v(1)
		\end{pmatrix}
	\end{align*}
	for $u \in H^1$. The generator of $S$ is the Laplace operator $\Delta$ with domain
	\begin{align*}
		\dom(\Delta) = \{u \in H^2(0,1): \, u'(0) = -u'(1) = u(0) + u(1) \}.
	\end{align*}
	The semigroup $S$ satisfies the assumptions of Theorem~\ref{thm:ev-pos} for the function $u = \one$; indeed, the first assumption follows from the fact that $S(t)$ maps $L^2(0,1)$ into the form domain $H^1(0,1)$ for each time $t$, and the second assumption can be shown by explicitly computing the resolvent of $\Delta$ at the point $0$ and by using a relation between the resolvent behaviour and the eigenfunction; we refer to \cite[Theorem~6.11]{Daners2016a} and \cite[Theorem~4.2]{DanersUnif} for details.
	
	Hence, we have $S(t) \ge 0$ for all sufficiently large times $t$.
\end{example}

The fact that the semigroup in the above example is not positive was first observed by Khalid Akhlil (private communication); he also gave a detailed treatment of the semigroup in \cite[Section~3]{Akhlil2018}. The semigroup appeared as an example of individual eventual positivity in \cite[Theorem~6.11]{Daners2016a} and of uniform eventual positivity in \cite[Theorem~4.2]{DanersUnif}.

\subsection{Schrödinger systems}

Plenty of further examples for eventual positivity can be found in \cite[Section~6]{Daners2016}, \cite[Section~6]{Daners2016a} and \cite[Chapter~11]{GlueckDISS}. In the following we give a further example which has not appeared in the literature, yet.

\begin{example}
	\label{ex:ev-pos-schroedinger-systems}
	Let $\emptyset \not= \Omega \subseteq \bbR^d$ be open, bounded and connected and assume that it has Lipschitz boundary. Let $V: \Omega \to \bbR^{N \times N}$ (for a fixed $N \in \bbN$) be a bounded and measurable mapping and assume that each matrix $V(x)$ is self-adjoint, negatively semi-definite and satisfies the conditions $\spec(V(x)) \cap i \bbR = \{0\}$ and $\ker V(x) = \linSpan \{c\}$, where $c \in \bbR^N$ is a vector that does not depend on $x$ and whose entries are all strictly positive. On the vector-valued Hilbert space $H = L^2(\Omega,\bbR^N)$ we consider the form $a: H^1(\Omega;\bbR^N) \times H^1(\Omega; \bbR^N) \to \bbR$ given by
	\begin{align*}
		a(u,v) = \sum_{k=1}^N \int_\Omega \nabla u_k \cdot \nabla v_k - \int_\Omega \langle Vu, v\rangle_{\bbR^N}
	\end{align*}
	for $u,v \in H^1(\Omega; \bbR^N)$. The associated semigroup $S$ on $H$ is contractive and self-adjoint and maps $H$ into $L^\infty(\Omega,\bbR^N)$ by an ultra contractivity argument. The largest eigenvalue of the generator $A$ is $0$, and the corresponding eigenspace is spanned by the vector $w = c \one$ (see \cite[Propositions~2.9 and~2.10]{Dobrick2020}). Hence, the assumptions of Theorem~\ref{thm:ev-pos} are satisfied for $u = (\one,\dots, \one)$. 
	
	Consequently, there exists a time $t_1 \ge 0$ such that $S(t) \ge 0$ for each $t \ge t_1$. However, the semigroup $S$ cannot be expected to be positive unless the off-diagonal entries of the matrices $V(x)$ are $\ge 0$.
\end{example}

\begin{further_references}
	The semigroup generator in Example~\ref{ex:ev-pos-schroedinger-systems} is a perturbation of the generator of a positive semigroup. In this context, it is worthwhile to mention that the relation between eventual positivity and perturbations is rather subtle; for the single operator case in finite dimensions this was observed in \cite{Shakeri2017}, and for $C_0$-semigroups in infinite dimensions a study of this topic can be found in \cite{DanersPERT}.
\end{further_references}


\end{document}